\documentclass[12pt]{amsart}
\usepackage[dvips]{graphics,graphicx,epsfig,color}
\usepackage{enumerate}
\usepackage{amssymb}
\usepackage[mathscr]{eucal}
\newtheorem{theorem}{Theorem}[section]
\newtheorem{lemma}[theorem]{Lemma}
\newtheorem{proposition}[theorem]{Proposition}
\newtheorem{corollary}[theorem]{Corollary}

\usepackage[utf8]{inputenc}
\usepackage{pgf,tikz}
\theoremstyle{definition}
\newtheorem{definition}[theorem]{Definition}

\newtheorem{example}[theorem]{Example}
\newtheorem{note}[theorem]{Note}

\theoremstyle{remark}

\numberwithin{equation}{section}

 \makeatletter
      \def\@setcopyright{}
      \def\serieslogo@{}
      \makeatother

\begin{document}

\title{A note on the $k$-defect number: \\
Vertex Coloring with a Fixed Number of Monochromatic Edges}

\author{$^\dag$E. Mphako-Banda, $^\ddag$C. Kriel and $^\star$A. Alochukwu }
\address{$^\dag$ $^\ddag$ School of Mathematics, University of the Witwatersrand, South Africa and \\ $^\star$Department of Mathematics, Computer Science and Physics, Albany State University, USA }

\email{$^\dag$Eunice.mphako-banda@wits.ac.za, $^\ddag$christo.kriel@wits.ac.za, $^\star$Alex.Alochukwu@asurams.edu}



\date{}

\maketitle

\begin{abstract}

In this paper, we introduce and study a novel graph parameter called the $k$-defect number, denoted $\phi_{k}(G)$, for a graph $G$ and an integer $0\leq k\leq |E(G)|$. Unlike traditional defective colorings that bound the local degree within monochromatic components, the $k$-defect number represents the smallest number of colors required to achieve a vertex coloring of $G$ having exactly \emph{$k$ monochromatic edges (also termed ``bad edges")}. This parameter generalizes the well-known chromatic number of a graph, $\chi (G)$, which is precisely $\phi _{0}(G)$. We establish fundamental properties of the $k$-defect number and derive bounds on $\phi _{k}(G)$ for specific graph classes, including trees, cycles, and wheels. Furthermore, we extend and generalize several classical properties of the chromatic number to this new edge-centric $k$-defect framework for values of $1\leq k\leq |E(G)|$. 
\end{abstract}

\section{Introduction}
Coloring problems in graphs are a classical yet vibrant area of research in discrete mathematics, with applications ranging from scheduling and frequency assignment to social and biological network analysis. The traditional chromatic number, $\chi(G)$, is a fundamental graph parameter defined as the minimum number of colors necessary to properly color the vertices of a graph $G$ such that no two adjacent vertices share the same color. This concept is extensively studied in graph theory and has broad applications, as detailed in (1). Associated with this is the chromatic polynomial, $\chi(G;\lambda)$, which counts the number of distinct proper vertex colorings of $G$ using $\lambda$ colors, with $\chi(G)$ being the smallest integer $\lambda$ for which $\chi(G;\lambda) \geq 0$. \\

Since the nineteenth-century four-color problem, results such as Brooks' theorem, bounds via the clique number $\omega(G)$, and the factorization of $\chi(G,\lambda)$, have become cornerstones of graph theory and its algorithmic and combinatorial applications \cite{GCPZ, TRB11}. While proper colorings demand that all adjacent vertices receive different colors, relaxations of this condition have given rise to various ``defective coloring" concepts in the literature, motivated by scenarios in which strict proper colorings are either unnecessary or infeasible. Notable examples, include
\begin{itemize}
\item \emph{Defective colorings}, where each colour class induces a subgraph of bounded maximum degree \cite{CB20,HM17};

\item \emph{Bounded-improper colorings}, which allow a bounded number of monochromatic adjacencies in each color class;

\item \emph{Fractional and list colorings} \cite{BP21}. 
\end{itemize}
Such models are particularly relevant in large-scale or noisy networks-such as wireless communication graphs or fault-tolerant distributed systems-where some ``colour conflicts" can be tolerated or even unavoidable.\\

A defective $(k,d)$-coloring refers to a $k$-coloring where each vertex has at most $d$ neighbors of the same color, focusing on local defectiveness within monochromatic components \cite{CWJ}. In contrast to these established notions, we introduce a new graph parameter, the $k$-defect number, \(\phi _{k}(G)\), which centers on the total number of monochromatic edges in a vertex coloring. \\

Our notion of a $k$-defect number differs from the ``vertex-coloring with defects" studied by Angelini et al. in \cite{ANG}, where each colour class is required to induce a subgraph of bounded maximum degree (a local defect bound). That is, their model fixes a local parameter $d$ and asks for the minimum number of colors so that in each color class the induced subgraph has maximum degree at most $d$ monochromatic edges. In contrast, we fix a global integer $k$ and seek the minimum number of colors so that exactly $k$ edges of $G$ are monochromatic.

A related but distinct direction concerns edge colorings. Casselgren and Petrosyan in \cite{INT} investigated improper interval edge colorings, where colors are assigned to edges so that at each vertex the set of incident edge colors forms an integer interval while allowing incident edges to share colors. Their notion of ``improper" concerns relaxed edge colorings with interval constraints.
Our proposed work, in contrast, develops a vertex-coloring analogue: we allow a prescribed number of edges to be monochromatic while counting the minimal number of colors required. \\

Specifically, for a given integer \(k\), \(\phi _{k}(G)\) is defined as the smallest number of colors needed to color the vertices of a graph \(G\) such that precisely \(k\) edges are monochromatic (i.e., connect vertices of the same color). An edge satisfying this condition is referred to as a "bad edge". This novel definition distinguishes our work by shifting the focus from local structural properties of monochromatic components to a global count of monochromatic edges. The $k$-defect number effectively generalizes the chromatic number, as \(\phi _{0}(G)\) is equivalent to \(\chi (G)\). The present work aims to explore this generalization systematically. \\

The structure of this paper is as follows: In Section \ref{Prel}, we outline some classic properties of the chromatic number, and a concise overview of the $k$-defect polynomial theory, essential for our analysis. 
This paper seeks to develop analogous results within the framework of the $k$-defect number, extending these insights for values of \(1\le k\le |E(G)|\). In Section \ref{PrelResults}, we formally define the $k$-defect number and present its fundamental properties. Section \ref{sec3} is dedicated to determining explicit $k$-defect numbers for specific graph classes, including trees, cycles, and wheels. In Section \ref{sec4}, we  establish some general results on the $k$-defect number, establishing relationships with other graph parameters and prove a necessary and sufficient condition for a graph to have all but one of the $k$-defect numbers equal to 2. We conclude in Section \ref{sec5} with open questions and avenues for future research.

\section{Preliminary results}\label{Prel}
In this section, we provide preliminary results on chromatic number of a graph and $k$-defect polynomials. We begin with the following Proposition that outlines some well known results on the chromatic number of a graph to which we will give analogous results in this study, followed by some facts on the $k$-defect polynomial of  a graph that are useful for this research. For proofs and further information, see \cite{HC69, EMB19}. \\

The  chromatic number, $\chi(G)$, is the least number of colors required  to color  a graph $G$  in such a way  that any pair of adjacent vertices  receives  different colors. The chromatic polynomial, $\chi(G;\lambda),$ is a function  which is associated with a graph $G$ and expresses the number of different  proper vertex colorings of  $G$  with $\lambda$ colors. The smallest integer $\lambda$  such that $\chi(G;\lambda) \geq 0$ is the chromatic number of the graph $G$.


\begin{proposition}
\label{Chrom}
Let $G$ be a graph and $\chi(G),$ the  chromatic number $G.$
\begin{enumerate}[(i)]
\item If $H$ is a subgraph of a graph $G$, then $\chi(H) \leq \chi(G).$
\item A graph $G$ of order $n$ has chromatic number 1 if and only if $G= \overline{K}_{n}$, the complement of a complete graph.
\item For every graph $G,$ $\chi(G) \geq \omega(G)$, where $\omega(G)$ is the clique number of $G$.
\item A non-empty  graph $G$ has chromatic number 2 if and only if  $G$ is bipartite.
 \end{enumerate}   
\end{proposition}

In an improper vertex coloring of a graph $G,$ we call  an edge   bad if it joins two vertices of the same color.  The $k$-defect polynomial, $\phi_{k}(G;\lambda),$ counts the number of ways of coloring $G$ using $\lambda$ colors such that it has $k$ bad edges. The $k$-defect polynomial,  $\phi_{k}(G;\lambda),$ can be computed using different methods just like the chromatic polynomial. The $0$-defect polynomial of a graph is just the chromatic polynomial. We refer the reader, to \cite{EMB19} for the basic theory of the $k$-defect polynomial.

The $k$-defect polynomial  of a graph
can be computed recursively as stated in  the following  proposition. 
\begin{proposition}
\label{theo333}
Let $G$ be a connected graph with  edge set $E$ and let $e \in E$. Then  
\begin{enumerate}[(a)]
\item $\phi_0(G;\lambda) = \chi(G;\lambda).$
\item for $1 \leq k < |E|$,
\[
\phi_k(G;\lambda) =
\begin{cases}
 \phi_{k-1}(G / e;\lambda), &  \text{if $e$ is a loop}\\
 \phi_{k-1}(G / e;\lambda) + (\lambda -1)\phi_{k}(G / e;\lambda),& \text{if $e$ is a bridge}\\
 \phi_k(G \backslash e ;\lambda) - \phi_k(G / e ;\lambda) + \phi_{k-1}(G / e;\lambda), &\text{otherwise.}
 \end{cases}\]
\item  for $k = |E|$,  $\phi_k(G;\lambda) = \lambda,$
\item  for $k > |E|$,  $\phi_k(G;\lambda) =0,$
\end{enumerate}
where $G\backslash e$ is the graph $G$ with edge $e$ deleted and $G/e$ is the graph after contraction of edge $e$.
 \end{proposition}

The rank,  $r(G)$, of a graph $G$, is defined as the number of vertices  minus  the number of components  of $G.$ A closed set $X$ of size $k,$ is the largest rank-$r$ subset of $E(G)$ containing $X.$  We denote the set of all closed sets  of $G$  by  $L(G).$ 

\begin{theorem}
\label{prop2}
Let $G$ be a graph. Then  $$\phi_{k}(G;\lambda) = \sum_{X \in L(G),\vert X \vert = k} \chi (G/X; \lambda)$$
if $G$ has at least one closed set  of size $k$. Otherwise $\phi_{k}(G;\lambda) = 0.$
\end{theorem}

\begin{proposition}\label{pemb19}
    Let $T_n$ and $C_n$ be a tree and a cycle graph of order $n$, respectively. Then the $k$-defect polynomial for $0 \leq k \leq |E|$ is, 
    \begin{enumerate}
        \item[(i)]  $\phi_k(T_n;\lambda)={{n-1}\choose {k}} \lambda(\lambda-1)^{n-1-k}.$
        \item[(ii)] $\phi_k(C_n;\lambda)={{n}\choose {k}}\left[(\lambda-1)^{n-k}+(-1)^{n-k}(\lambda-1)\right].$
    \end{enumerate}
\end{proposition}

\section{Main Results: The $k$-defect number}\label{PrelResults}
In this section we introduce the $k$-defect number and present some basic results on this parameter.
\begin{definition}
For a coloring $c:V(G)\rightarrow \{1,\ldots,\ell\}$, a vertex $u \in V(G)$ and a color~$t$, $1\leq
t\leq \ell$, we denote by $c(u)$ the color of vertex $u$. That is, $c(u) = t$ means that $t$ is the color assigned to vertex $u$.
\end{definition}

\begin{definition}
Let $G$ be a  vertex-colored graph. An edge $uv \in E(G)$ is called bad if $c(u) = c(v)$.  
\end{definition}

\begin{definition}
The $k$-defect number, denoted $\phi_k(G)$ of graph $G$ is the smallest number of colors needed to color a graph $G$ with $k$ bad edges. Otherwise, if it is not possible  to color a graph $G$  with $k$ bad edges, then $\phi_k(G)=0.$  Thus, just like the chromatic number and the chromatic polynomial, the $k$-defect number, $\phi_k(G),$ is the smallest integer that gives  a positive value of  the $k$-defect polynomial, $\phi_{k}(G;\lambda).$ 
\end{definition}

We follow Proposition~\ref{theo333}  and  note that the $k$-defect number of a graph  is only defined  for $k$  such that  the $k$-defect polynomial is  greater than 0. In other words, the $k$-defect number is defined if it is possible to color a graph $G$  with $k$ bad edges.

 We state the following  basic properties of the $k$-defect number without proof.
\begin{proposition} Let $G$ be a graph with $m$ edges, then $\phi_m(G)=1.$ 
\end{proposition}
\begin{proposition} Let $H$ be subgraph of $G.$  Then $\phi_k(H)\leq \phi_k(G)$  for  $0\leq k\leq |E(G)|.$
\end{proposition}
\begin{proposition}\label{ChroG}
A graph $G$ of order $n$ has the $k$-defect number $\phi_k(G)=1$ if and only if $G= \overline{K}_{p} \cup H,$ $|E(H)|=k$  and $p+|V(H)|= n.$
\end{proposition}

\section{$k$-defect number for some classes of graphs}
\label{sec3}
In this section, we present  explicit   $k$-defect  numbers for some classes of graphs, namely: trees, cycles and wheels. A straight forward method to find explicit 
 $k$-defect  number of a graph $G$ would be using it's  $k$-defect polynomial. As stated in \cite{EMB19},  it is not easy  to  find explicit expressions of the $k$-defect polynomial for a class of graphs, hence in this section, we prove the results using  the $k$-defect polynomial if it is known, otherwise we  use direct reasoning  on  coloring. 

 For the proof of the results on trees and cycles on $n$ vertices, we make use of known results on the $k$-defect polynomials in Proposition~\ref{pemb19}.

\begin{theorem}\label{kdefCyc}
Let $C_n$  be a cycle graph on $n$ vertices and $1 \leq k \leq n-2,$ then 
\begin{eqnarray*}
\phi_{k}(C_n) &=& \begin{cases}  2 &\text{ if   $n-k$ is even}\\  
 3 & \text{ if $n-k$ is odd.}\\
\end{cases}
\end{eqnarray*}

\end{theorem}
\begin{proof}
The proof follows from the $k$-defect polynomial in Proposition \ref{pemb19}. It suffices to observe that
\[
\phi_k(C_n;3) = {n \choose k}(2^{n-k} + (-1)^{n-k}\cdot 2) > 0,
\] for all $1 \leq k \leq n-2$. Hence, we need at most $3$ colors to color the graph for any  $1 \leq k \leq n-2$ and we conclude $\phi_k(C_n) \leq 3$.

Meanwhile
\[
\phi_k(C_n;2) = {n \choose k}(1 + (-1)^{n-k})
\] which is $0$ if and only if $n-k$ is odd and greater than $0$ otherwise. Thus $\phi_k(C_n)=3$ if $n-k$ is odd and $\phi_k(C_n)=2$ if $n-k$ is even.  
\end{proof}

\begin{theorem}\label{kdefTre}
 Let $T_n$ be a tree on $n$ vertices and $0 \leq k \leq n-2,$ then $\phi_k(T_n) =2.$ 
\end{theorem}

\begin{proof}
    Following a similar argument as in the proof of Theorem~\ref{kdefCyc}, and making use of the $k$-defect polynomial  in Proposition \ref{pemb19}, the result holds.
\end{proof}

\begin{lemma}\label{lem:wheelmindefectscol2}
Let $W_n$ be a wheel on $n>3$ vertices and $m=2n-2$ edges. The minimum number of bad edges possible with a $2$-coloring of $W_n$ is $\lfloor \frac{n}{2}\rfloor$.
\end{lemma}

\begin{proof}
    We consider the two cases where $n$ is odd and $n$ is even.

    Let $n$ be odd then the outer cycle is an even cycle. Let $v_1, \dots, v_{n-1}$ denote the vertices on the outer cycle, and $z$ denote the center of the wheel. We color adjacent vertices $v_1, \dots, v_{n-1}$ by alternating two colors $1$ and $2$, giving a proper coloring of the cycle, with $\frac{n-1}{2}$ vertices colored $1$ and $\frac{n-1}{2}$ vertices of color 2. Now assigning $z$ either color $1$ or color $2$ gives $\frac{n-1}{2}$ bad edges, so two colors suffice to give $\frac{n-1}{2}$ bad edges.

    To show that this is indeed the minimum number of bad edges we start by coloring $z$ with color $1$ and color the vertices on the cycle in such a way that we add the minimum number of bad edges. We can add bad edges on the cycle by coloring adjacent vertices color $2$, on the spokes (edges of the form $zv_i$) by alternately coloring vertices on the cycle with $1$ and $2$, or a combination of the two. We note that coloring adjacent vertices color $2$ on the cycle always adds a bad edge, while coloring the cycle to get bad edges on the spokes adds one bad edge half the time and zero bad edges the other half. Hence, we have fewer bad edges with two colors when we color the cycle properly. By the same reasoning this will also give the same or fewer bad edges than when we choose a combination of cycle edges and spokes as bad edges.

Let $n$ be even then the outer cycle is odd and we have at least one bad edge with a $2$-coloring. color two adjacent vertices on the cycle with color $1$ and the remaining vertices properly by alternating colors $1$ and $2$. We have $\frac{n}{2}$ vertices colored $1$ and $\frac{n-2}{2}$ vertices colored $2$. Assign $z$ color $2$ to minimise bad edges and we see that two colors suffice for $\frac{n-2}{2}+1=\frac{n}{2}$ bad edges.

To show that this is the minimum, we contract the bad edge on the cycle. This results in an odd wheel with the coloring that we showed gives the minimum number of bad edges. Adding the edge back means that we added exactly one bad edge after adding one vertex, and hence $\frac{n}{2}$ is indeed the minimum number of bad edges possible with two colors.
\end{proof}

\begin{theorem}
Let $W_n$ be  a wheel on $n>3$ vertices and $m=2n-2$ edges.
The $k$-defect number of $W_n$
\begin{eqnarray*}
\phi_k(W_n) &=& \begin{cases}0 & \text{ if $2n-3 \leq k \leq 2n-4$}\\
 1 & \text{ if $k =2n-2$ } \\
2 & \text{ if $\lfloor \frac{n}{2}\rfloor \leq k  \leq 2n-5$}\\
3  &\text{ if $1 \leq k <\lfloor  \frac{n}{2}\rfloor $}\end{cases}\\
\end{eqnarray*}
\label{thm:kdefectwheel}
\end{theorem}

\begin{proof}
We note from \cite{KM22} that $\phi_k(G)=0$ for all $|E|-\lambda < k < |E|$, where $\lambda$ is the minimum edge-cut set number of a graph $G$. Since $\lambda(W_n)=3$ we have $\phi_k(W_n)=0$ if $2n-3 \leq k \leq 2n-4$.

If $k=m=2n-2$ we have all the edges bad and that requires that we color all the vertices the same color, so $\phi_{2n-2}(W_n)=1$.

For $\lfloor \frac{n}{2}\rfloor \leq k  \leq 2n-5$ and $n$ odd, we start with the coloring from Lemma \ref{lem:wheelmindefectscol2}, with $z$ colored $1$, that gave the minimum number of bad edges on two colors. We successively re-color each vertex of color $1$ on the cycle with color $2$. Every time we re-color a vertex, we lose one bad edge (a spoke) and add two bad edges (on the cycle), effectively increasing $k$ by one. We do this until all the edges on the cycle are bad, that is $k=n-1$. Now re-color two adjacent vertices on the cycle with color $1$. We lose three bad edges on the cycle but add three bad edges (two spokes and one on the cycle). We still have $n-1$ bad edges. Now re-coloring a vertex on the cycle adjacent to one of the cycle vertices colored $1$ with color $1$, we add two bad edges (one spoke and one on cycle) and lose one bad edge on the cycle, thus increasing $k$ by one. We can continue to do this until we have one vertex colored $2$ on the cycle and $k=2n-5$. The remaining three edges are covered by the previous two cases. Using a similar coloring argument for even $n$ we get the same result.

For $1 \leq k <\lfloor  \frac{n}{2}\rfloor $ and $n$ odd, we again start with the coloring from Lemma \ref{lem:wheelmindefectscol2}. Re-color one of the vertices of color $1$ with color $3$. This reduces the number of bad edges by one. Continue doing this until there is one vertex of color $1$ left on the cycle and $k=1$. Since we know that $k =\lfloor  \frac{n}{2}\rfloor $ is the minimum number of edges we can get with two colors, the result follows. Once again, a similar re-coloring argument holds for even $n$.
\end{proof}
\begin{note}
It is interesting to observe that $\phi_0(W_n) = \chi(W_n)$ depends on $n,$ specifically, $\phi_0(W_n) = 4$ when $n$  is even and  $\phi_0(W_n) = 3$ when $n$ is odd. However, for $1\leq k < 2n-3$, $\phi_k(G)$ is constant on the given intervals, independent of the parity of $n$.
\end{note}

\section{Some general results on the $k$-defect number}
\label{sec4}
In this section, we give some general results on the $k$-defect number of graphs. As a by-product of our investigation, these general results shows the relationship between the $k$-defect number and how the $k$-defect chromatic number of a graph can be computed in terms of other graph parameters.

We know that the chromatic polynomial  $\chi (G; \lambda)$ counts the number of ways of coloring $G.$ The following Lemma is a well known fact on chromatic polynomials.
\begin{lemma}
\label{lem1}
 The chromatic polynomial of $G$ can be factorised as 
 $\chi (G; \lambda)= \lambda (\lambda -1)(\lambda -2) \cdots (\lambda -r)[Q(\lambda)]$ such that $\chi(G)=r+1$  and $Q(\lambda)$ is a polynomial without positive integer roots.   
\end{lemma}
\begin{theorem}The $k$-defect number $\phi_k(G) = \chi(G/X)$, where $\chi(G/X)$ is the  minimum chromatic number  over all minors of $G$ obtained by contracting closed sets of size $k.$
\label{th:kdefectcontr}
\end{theorem}
\begin{proof} 
By Theorem~\ref{prop2} we have,
$\displaystyle \phi_{k}(G;\lambda) = \sum_{X\in L(G),\vert X \vert = k} \chi (G/X; \lambda)$ if $G$ has at least one closed set  of size $k.$ Thus,  if $L(G) = \{X_1, X_2, \cdots , X_n\}$ then we can rewrite 
\begin{eqnarray*}
   \phi_{k}(G;\lambda) &=& \sum_{X_i} \chi(G/X_i; \lambda).\\
   &=&  \sum\lambda(\lambda -1)(\lambda -2) \cdots (\lambda -r_i)  Q_i(\lambda) \text{ by Lemma~\ref{lem1} }.\\
\end{eqnarray*}

Hence, we can rewrite
\begin{eqnarray*}
   \phi_{k}(G;\lambda) 
   &=& \lambda(\lambda -1)(\lambda -2) \cdots (\lambda -r) T_i(\lambda) Q_i(\lambda)
\end{eqnarray*}
where $r$ is the smallest $r_i$ and $T_i(\lambda) = (\lambda -r_i+1) \cdots (\lambda -r_i+t)$ for some $t \geq 0.$

\end{proof}

\begin{corollary}
Let $H_1, H_2, \cdots, H_q$ be closed sets of $G$ of size $k.$ Then $\phi_k(G)= \chi(G/H)$ where  $\omega (G/H_i)$ is the minimum  over $H_1, H_2, \cdots, H_q .$ 
\end{corollary}

\begin{proof}
	This follows directly from Theorem \ref{th:kdefectcontr} and Proposition \ref{Chrom} (iii).
\end{proof}

\begin{corollary}The $k$-defect number of $K_n$ is $\chi(K_n /H)$, where $H$ is a closed set of size $k$ of maximum order.
\end{corollary}
\begin{proof}
	Closed sets of size $k$ in $K_n$ are disjoint unions of complete graphs, see \cite{KM22}. Hence, contracting a closed set $H$ of size $k$ of maximum order will give $\omega(G/H)$ where $\omega (G/H_i)$ is the minimum  over $H_1, H_2, \cdots, H_q .$
\end{proof}

Recall that a set of vertices is \emph{independent} if none of the vertices in the set are adjacent in the graph.

\begin{theorem}
Let $G$ be a bipartite graph of size $m$, then $\phi_k(G)=2$ for all $0\leq k \leq (m-1)$ if and only if there exists an independent subset $I\subseteq V(G)$ of vertices such that
\[\sum_{v\in I} \textrm{deg}(v)=k,
\]
for all $1\leq k \leq (m-1).$
\label{th:k2iff}
\end{theorem}
\begin{proof}
Since $G$ is bipartite $\phi_0(G)=2$. Also, we have $\phi_m(G)=1$ for all graphs.

$(\implies)$ We use two colors to color $G$ properly. Since $G$ is bipartite, changing the color of a vertex $v$ of color $1$ to color $2$ means that all edges incident on $v$ become bad edges. It follows that if we have an independent subset of vertices $v$ such that $\sum \textrm{deg}(v)=k$ we will have $\phi_k(G)=2$.

$(\Longleftarrow)$ Suppose there is no such $I$, then it its possible to color adjacent vertices resulting in $k$ bad edges. Without loss of generality, let $u$ and $v$ be adjacent such that $\textrm{deg}(u)+\textrm{deg}(v)=k$. color $u$ the same as $v$. We have $\textrm{deg}(u)$ bad edges. But edges incident on $v$, not incident on $u$, can only become bad if we re-color a third vertex $w$ adding $\textrm{deg}(w)$ bad edges, leading to a contradiction.
\end{proof}

\begin{example}
	Consider the induced $K_{3,4}$ in Figure \ref{fig:egk=2}, the subgraph with solid edges and partite sets $v_i$ and $u_i$ colored white and black respectively. Clearly we can have three, six, nine or twelve bad edges by changing the color of the required number of white vertices. Similarly, we can have four, eight or twelve bad edges by changing the color of the required number of black vertices. We cannot have one, two or five bad edges since we do not have vertices of degree one or two. However, we can also not have seven, ten or eleven bad edges, even though we do have vertex degrees that add up to these numbers. That is because every vertex of degree four is adjacent to every vertex of degree three. If we consider the entire graph, however, it is possible to have $k$ bad edges for all $1\leq k\leq  14$ using the two colors, since we have independent vertex sets with degree sums for all the required $k$.
\end{example}

\begin{figure}[!ht]
	\centering
	\includegraphics[scale=1.5]{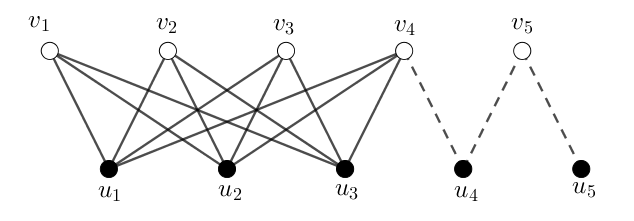}
		\caption{}
		\label{fig:egk=2}
\end{figure}

\section{Conclusion and Questions}
\label{sec5}

\begin{sloppypar}
It is interesting to note that, given enough colors, it is always possible to color a graph properly, that is $\phi_0(G)$ is defined for all graphs. However, for $k\geq 1$ we have cases where $\phi_k(G)=0$, no matter how many colors are available. We know for example from Theorem \ref{th:k2iff} what these values for $k$ would be for bipartite graphs. We also know that for complete graphs of order $n$ and  for integers $p$ and $k,$ ${\displaystyle 1\leq p \leq \lfloor \frac{-1+\sqrt{8n-15}}{2} \rfloor}$, $\phi_k(K_n)=0$ on the intervals  
$\displaystyle\binom{n-p}{2}+\displaystyle\binom{p}{2} < k < \displaystyle\binom{n-p+1}{2},$ see \cite{KM22}. On the other hand, for trees, $\phi_k(T_n)\neq 0$ for any $k$. Is it possible to determine other graph parameters that cause the number of these zeroes to increase?
\end{sloppypar}

When we consider all the possible values of $k$  for the $k$-defect number, it  is complex to generalize the results.  Hence an approach of studying  each value of $k$ separately seems reasonable in order to get a better understanding of the $k$-defect number.  Following the theory on chromatic numbers of a graph we pose further questions: can we find upper and lower bounds for the $k$-defect numbers for different values of $k$? The chromatic number of a graph is a widely used graph parameter, can some of the $k$-defect numbers get the same status?

\section*{Acknowledgments}
Special thanks to  Paul Horn (University of Denver, USA)  for valuable suggestions and discussions in improving this paper.

This research is supported in part by NGA(MaSS) grant number UCDP-712,
CoE(MaSS) grant number 2022-044-GRA-Workshop,
NITheCS from the Mathematical Structures Research grant,
UJ GES 4IR.


\begin{thebibliography}{99}
\bibitem{ANG} P. Angelini, M. Bekos, F. De Luca, W. Didimo, M. Kaufmann, S. Kobourov, F. Montecchiani, C. Raftopoulou, V. Roselli, and A. Symvonis, Vertex-coloring with defects. \textit{Journal of Graph Algorithms and Applications}, 21(3), pp.313-340, 2017.

\bibitem{BP21} S. Cambie, W. Cames van Batenburg, E. Davies, and R. J. Kang, Packing list-colorings, Random Structures \& Algorithms 64 (2024), no. 1, 62–93. \\ 

\bibitem{CB20} L. J. Cowen, R. H. Cowen, and D. R. Woodall,
Defective colorings of graphs in surfaces: partitions into subgraphs of bounded valency, J. Graph Theory 10 (1986), 187–195. \\ 

\bibitem{INT} C.J. Casselgren, and P.A. Petrosyan, Improper interval edge colorings of graphs. \textit{Discrete Applied Mathematics}, 305, pp.164-178, 2021.

\bibitem{GCPZ} G. Chartrand and P. Zhang. Chromatic graph theory. In K.H. Rosen, editor, Discrete Mathematics and Its Applications,  147--203 Chapman \& Hall, New York, USA,2009. \\

\bibitem{HC69}H.H. Crapo, The  Tutte polynomial, Aequationes Math.,3 (1969) 211-229.\\

\bibitem{CWJ} L.J. Cowen, W. Goddard, and C. E. Jesurum. Coloring with defect. In SODA, vol. 97, pp. 548-557. 1997.

\bibitem{FH}F. Harary, Graph Theory, Addison-Wesley, Reading MA, (1969).\\




\bibitem{TRB11} T.R. Jensen and B. Toft, Graph Coloring Problems. \textit{Wiley}, 2011.\\

\bibitem{KM22} C. Kriel and E. Mphako-Banda, Sizes of flats of cycle matroids of complete graphs, \textit{Revista Colombiana de Matem\'aticas} \textbf{56(1)} (2022), 63--75.\\

\bibitem{EMB19} E. Mphako-Banda, An introduction to the $k$-defect polynomials, \textit{Quaestiones Mathematicae} \textbf{42(2)} (2019), 207--216.\\

\bibitem{HM17}  P. Ossona de Mendez, S. Oum, and D. R. Wood,
Defective coloring of graphs excluding a minor,
J. Combin. Theory Ser. B 114 (2015), 1–48. \\ 

\bibitem{RR68} R.C. Read, An introduction to chromatic polynomials, \textit{Journal of Combinatorial Theory}, vol. 4, 52-71, 1968.\\

\end{thebibliography}
 \end{document}